\documentclass[12pt]{article}
\usepackage{amsmath,amsfonts,amssymb,amsthm,amscd}
\title{On maximal stable quotients of definable groups  in $NIP$ theories}
\date{\today}
\author{Mike Haskel and Anand Pillay\thanks{Partially supported by NSF grant DMS-1360702}\\University of Notre Dame}
\newtheorem{Theorem}{Theorem}[section]

\newtheorem{Definition}[Theorem]{Definition} 
\newtheorem{Remark}[Theorem]{Remark}
\newtheorem{Lemma}[Theorem]{Lemma}
\newtheorem{Corollary}[Theorem]{Corollary}

\newtheorem{Problem}[Theorem]{Problem}

\begin{document}
\maketitle

\begin{abstract} For  $G$ a group definable in a saturated model of a $NIP$ theory $T$, we prove that there is a smallest type-definable subgroup $H$ of $G$ such that the quotient $G/H$ is stable. This generalizes the existence of $G^{00}$, the smallest type-definable subgroup of $G$ of bounded index. 
\end{abstract}

\section{Introduction and preliminaries}
Fix an arbitrary theory $T$ and saturated model $\bar M$.  If $G$ is a definable group and $A$ a set of parameters over which $G$ is definable, then there is always a smallest subgroup of $G$ which is type-definable over $A$ and has bounded index in $G$. This group is called $G_{A}^{00}$, and is a normal subgroup.   Moreover $G/G^{00}_{A}$  has naturally the structure of a compact (Hausdorff) topological group. 
When $T$ has $NIP$ it was proved (\cite{Shelah},  \cite{NIPI}) that $G^{00}_{A}$  does not depend on $A$, so the compact quotient $G/G^{00}$ is a canonical invariant of $G$,  independent of any choice of parameters. This paper is motivated partly by the question whether there are other interesting canonical quotients of a  definable group $G$ in the $NIP$ environment.  Feeding into this is the example of $ACVF$ where the additive group of the residue field $k$ is a stable quotient of the additive group of the valuation ring $V$ and is the maximal such.  Moreover we have ``domination" of $V$ by the $k$, again analogous to the compact domination of $G$ by $G/G^{00}$ in some cases.  

So we are interested in {\em stable} quotients of a definable group. By {\em stability} of a definable set $X$ in an ambient (saturated) structure, we mean that every complete type over a model, which  extends  $X$, is definable. This notion is sometimes called ``stable, stably embedded". 
The notion extends to type-definable sets and there are equivalent characterizations in terms of the order property (see for example the appendix to  \cite{CH}, where  actually the expression ``fully stable" is used).  However the group quotients we are interested in may very well be by a type-definable subgroup $H$ in which case $G/H$ is hyperdefinable. By a hyperdefinable set we mean precisely the quotient $X/E$ of a (type-)definable set by a type-definable equivalence relation. So we have to make sense of the stability of a hyperdefinable set in an arbitrary ambient theory, and we give a suitable definition below (Definition 1.3).

The main result is: 
\begin{Theorem} Suppose $T$is $NIP$, and $G$ is a definable group, defined over $\emptyset$ in a saturated model of $T$. Then there is a type-definable subgroup of $G$ which we call $G^{st}$,  with the following characteristic properties:
\begin{enumerate}
\item $G^{st}$ is type-definable over $\emptyset$,
\item $G/G^{st}$ is stable, and 
\item $G^{st}$ is the smallest type-definable (with parameters) subgroup $H$ of $G$ such that $G/H$ is stable. 
\end{enumerate}
Moreover $G^{st}$ is a normal subgroup. 
\end{Theorem} 

The proof also works when $G$  itself  is type-definable.
We think of $G^{st}$ as the {\em stably connected component} of $G$, noting that $G^{st}$ will have no  proper hyperdefinable stable quotient.  Actually we could and should distinguish between $G^{st,0}$ and $G^{st,00}$ where $G^{st,0}$ is the intersection of all definable subgroups of $G$ for which the quotient is stable, and $G^{st,00}$ which is what we have called $G^{st}$.  We will discuss some issues and questions around this at the end of the paper.  In any case our proof of Theorem 1.1 in the next section will be an adaptation of the proof  in \cite{NIPI}, Proposition 6.1,  of  Shelah's result on the existence of $G^{00}$ in $NIP$ theories  \cite{Shelah}.

Our notation is standard. $T$ will denote a complete theory, and we typically work in a very saturated model ${\bar M}$ of $T$ (which we may sometimes want to enlarge).  ``Small" or ``bounded" means of cardinality strictly less that the degree of saturation of the monster model.  Definable, type-definable etc. usually means with parameters which we sometimes exhibit (as in ``$A$-definable"  or ``over $A$"). We feel free to use standard techniques in model theory without further explanation. 

The notion of boundedness is a bit delicate. For example suppose $X_{a}$ is a set which is type-definable over a possibly infinite tuple $a$. Then if we say that the intersection of the set of conjugates $\{X_{a'}: tp(a') = tp(a)\}$ is a bounded subintersection, we mean  that in some model $M$ there is a set $A$ of realizations of $tp(a)$ such that for any elementary extension $M'$ of $M$  the intersection of $\{X_{a'}: a'\in M', tp(a') = tp(a)\}$ is equal to $\cap_{a'\in A}X_{a'}$.  Although we could easily compute such a bound if it exists, we will not worry about it.  In practice, $a$ may be a countable tuple, and we want unboundedness to enable us to apply Erd\"os-Rado,  namely extract indiscernibles, so fixing in advance a suitable degree of saturation of the monster model $\bar M$ will be enough for us to work only in ${\bar M}$. 

Also we will assme $T = T^{eq}$.

\begin{Definition} By a hyperdefinable set we mean something of the form $X/E$ where $X$ is type-definable, and $E$ is a type-definable equivalence relation on $X$. If both $X,E$ are (type-)defined over $A$ we say $X/E$ is defined over $A$.
\end{Definition}

We usually take $X$ to be a subset of a sort $S$  in $T^{eq}$ (although   the definition makes sense and will be used  when $X$ is so-called $*$-definable, namely is a set of possibly infinite tuples).  It can be easily seen that the equivalence relation $E$ on $X$ is the restriction to $X$ of
some equivalence relation $E'$ on $S$, type-defined over the same set of parameters as $E$ and $X$.  Moreover we can rechoose $X$ to be  a union of $E'$-classes (and still type-definable over the same parameters as $X$ and $E$).  So hyperdefinable sets are ``type-definable" subsets of hyperdefinable sets of the form $S/E$ ($S$ a sort). Assuming $E$ is over $A$, elements of $S/E$ are called hyperimaginaries over $A$, and basic notions of model theory (types, indiscernibles, \ldots) were extended to hyperimaginaries in \cite{HKP} with  a nice account in \cite{Casanovas-book}  for example.

The notion of a type-definable subset of a hyperdefinable set $X/E$  is obvious (namely $X'/E$ for some type definable subset of $X$). Also a possibly infinite Cartesian product of hyperdefinable sets is a hyperdefinable set.

\begin{Definition}  Assume $X/E$ is over $A$. Then $X/E$ is stable (relevant to the ambient theory $T$)  if for all  sequences $(a_{i},b_{i})_{i<\omega}$ which are indiscernible over  $A$, and with $a_{i}\in X/E$ for all (some) $i$,  we have that $tp(a_{i}, b_{j}/A) = tp(a_{j},b_{i}/A)$ when $i\neq j$. 
\end{Definition} 

In the current paper we will be directly using Definition 1.3, which note agrees with the established notions for stability of type-definable sets in an arbitrary ambient theory. Note also that if $E$ is  bounded (i.e. has boundedly many classes) then $X/E$ is stable. 

It is easy to see that:
\begin{Remark}  Stability of hyperdefinable sets is preserved under (possibly infinite) Cartesian products and taking type-definable subsets. 
\end{Remark}

By a hyperdefinable group, we will mean a hyperdefinable set $G$ (in a saturated model of $T$) equipped with a group operation  which is a type-definable subset of $G\times G\times G$. In this paper we will be mainly concerned with hyperdefinable groups (or homogeneous spaces) of the form $G/H$ where $G$ is a definable group and $H$ a type-definable subgroup.

\section{Proof of  Theorem 1.1}

Recall that the $NIP$ chain condition on definable subgroups, says that if $G$ is a definable group in an $NIP$ theory and  $\cal H$ is a uniformly definable family of definable subgroups of $G$, then there is $N<\omega$ such that any intersection of finitely many members of $\cal H$ is a subintersection of at most $N$ elements of $\cal H$.  The first lemma is an analogue for type-definable subgroups, which is easily seen to imply the definable case mentioned above.
\begin{Lemma} Assume $T$ has $NIP$, and $G$ is a definable group, defined over $A$ say.   Suppose that $H_{a_{0}}$ is a type-definable subgroup of $G$, defined over the tuple $a_{0}$. Then there do not exist an $A$-indiscernible sequence $((a_{i},g_{i}):i < \omega)$ such that 

$$ g_{i}\in H_{a_{j}} \leftrightarrow i\neq j$$

for all $i,j <\omega$. 

\end{Lemma}
\begin{proof} Suppose otherwise. Without loss of generality $A=\emptyset$. 
Note that for any $y_{1}, y_{2}\in H_{a_{0}}$, $y_{1}g_{0}y_{2}\notin H_{a_{0}}$ (as $g_{0}\notin H_{a_{0}}$). So by compactness there is a formula $\phi(y,a_{0})$ in the partial type defining $H_{a_{0}}$ such that if $d_{1},d_{2}\in H_{a_{0}}$ and $d = d_{1}g_{0}d_{2}$ then 
$\models \neg\phi(d,g_{0})$. 

\vspace{2mm}
\noindent
By indiscernibility, for each $i<\omega$, $\phi(y,a_{i})$ has the same property with respect to $H_{a_{i}}$.  Namely
\newline
(*) whenever $d_{1}, d_{2}\in H_{a_{i}}$ and $d = d_{1}g_{i}d_{2}$ then 
$$\models \neg\phi(d,a_{i})$$

\vspace{2mm}
\noindent
We now show that the formula $\phi(y,z)$ has the independence property. Let $I$ be any finite subset of $\omega$, and let $d_{I}\in G$ be the group product  $\prod_{i\in I}g_{i}$.  Then by (*) and the fact that $g_{i}\in H_{j}$ whenever $i\neq j$, we see that for all $j<\omega$,  we have
$\models \phi(d_{I},a_{j})$ iff  $j\notin I$. 

So $T$ has the independence property. Contradiction. 
\end{proof}

\begin{Corollary} Let again $T$ have $NIP$,  $G$ be a definable group over $A$, and $H_{a_{0}}$ a type-definable subgroup of $G$, defined over $a_{0}$.
 Then for any sufficiently large family $\{H_{a_{\alpha}}:\alpha < \kappa\}$ of $A$-automorphic conjugates of $H_{a_{0}}$, there is $\beta<\kappa$ such that 
$H_{a_{\beta}}\supseteq \cap_{\alpha\neq\beta}H_{a_{\alpha}}$. 
\end{Corollary}
\begin{proof} Let us write $H_{\alpha}$ for $H_{a_{\alpha}}$.   If the corollary fails, let $g_{\beta}$ be chosen for each $\beta<\kappa$ such that $g_{\beta}\in H_{\alpha}$ for all $\alpha\neq\beta$, but $g_{\beta}\notin H_{\beta}$.  By Erd\"os-Rado we can rechoose an indiscernible (over $A$) sequence $(a_{i},g_{i})_{i<\omega}$ such that $g_{i}\in \cap_{j\neq i}H_{j}\setminus H_{i}$ for all $i$, contradicting Lemma 2.1. 
\end{proof}

\begin{Remark}  Let us say that a hyperdefinable set $X/E$ has $NIP$ if there do not exist an indiscernible sequence $(b_{i}:i<\omega)$ and $d\in X/E$ such that $((d,b_{2i}, b_{2i+1}):i<\omega)$ is indiscernible and $tp(d,b_{0}) 
\neq tp(d,b_{1})$.   (Note that the $b_{i}$ can be anywhere, not necessarily in $X/E$.) Then  the proofs above can be modified to yield the same conclusions for  $G$  hyperdefinable with $NIP$. 

\end{Remark}

\begin{Lemma} Suppose $T$ has $NIP$, and let $G$ be a definable group (over $\emptyset$). Let $H$ be a type-definable subgroup of $G$ such that $G/H$ is stable.  Then the intersection of all conjugates (under $A$-automorphisms) of $H$ is a bounded subintersection. 
\end{Lemma}
\begin{proof}  Let $H$ be type-defined over the sequence $a$. If the lemma fails then we can find arbitrarily long sequences $(H_{a_{\alpha}}:\alpha < 
\kappa)$ with $tp(a_{\alpha}) = tp(a)$ for all $\alpha$ and such that  $H_{\beta}$ does not contain $\cap_{\alpha < \beta} H_{\alpha}$ for all $\beta < 
\kappa$. Fix some long such sequence, and let $g_{\beta}\in (\cap_{\alpha<\beta}H_{\alpha})\setminus H_{\beta}$ for each $\beta$.  By Erd\"os-
Rado,  we may and will assume that the sequence $((a_{\alpha}, g_{\alpha}):\alpha < \kappa)$ is indiscernible.
Now by Corollary 2.2, there is $\beta <\kappa$ such that $H_{a_{\beta}}\supseteq \cap_{\alpha\neq\beta} H_{\alpha}$. 
Now we can clearly ``stretch" the indicernible sequence  $(a_{\alpha},g_{\alpha})_{\alpha<\kappa}$ by replacing $(a_{\beta},g_{\beta})$ by a sequence $((b_{i},h_{i}):i<\omega)$.  
Namely, the  sequence $(a_{\alpha},g_{\alpha})_{\alpha < \beta}$ followed by the sequence $(b_{i},h_{i})_{i<\omega}$ followed by the sequence $(a_{\alpha},g_{\alpha})_{\beta < 
\alpha < \kappa}$ is indiscernible, 
and  of course $tp(b_{i},h_{i}/\{a_{\alpha},g_{\alpha}: \alpha < \kappa, \alpha\neq \beta\}) = tp(a_{\beta},g_{\beta}/\{a_{\alpha},g_{\alpha}: \alpha < 
\kappa, \alpha\neq \beta\})$ for each $i<\omega$. Let $H = \cap_{\alpha\neq\beta}H_{\alpha}$, and write $H_{i}$ for $H_{b_{i}}$.  So we have
\newline
(I) $H\subseteq H_{i}$ for all $i<\omega$, and
\newline
(II) $G/H$ is stable  (by Remark 1.4), and also of course
\newline
(III) $h_{i}\in \cap_{j<i}H_{j}\setminus H_{i}$  for all $i<\omega$. 

Let us work over the set of parameters $A = \{a_{\alpha}:\alpha\neq \beta\}$ over which $H$ is defined. Note that $(b_{i},h_{i})_{i<\omega}$ is still indiscernible over $A$, as is the sequence
$(b_{i},(h_{i}/H))_{i<\omega}$. 
So as $h_{i}/H \in H_{j}/H$ for $j<i$ in $\omega$ it follows from Definition 1.3 and (II) that $h_{i}/H \in H_j/H$ for $i<j$ in $J$.  By (I) it follows that $h_{i}\in H_j$ for all $i<j$ in $\omega$. We conclude that

$$ h_{i}\in H_{j} \leftrightarrow i\neq j$$
for all $i,j\in \omega$.

As we are assuming that $T$ has $NIP$, this is a contradiction to Lemma 2.1.
\end{proof}

\begin{proof}[Proof of Theorem 1.1]   
If $H_{a}$ is a subgroup of $G$, type defined over $a$, such that $G/H_{a}$ is stable, then the intersection of all automorphic conjugates of $H_{a}$ is, by Lemma 2.3, a bounded subintersection $H$ which is clearly $\emptyset$-invariant hence type-defined over $\emptyset$. Moreover $G/H$ is also stable, by Remark 1.4.  In particular $H_{a}$ contains $G^{st}_{\emptyset}$, the smallest type-definable over $\emptyset$ subgroup of $G$ with stable quotient. This gives (i), (ii), and (iii) of Theorem 1.1. 

Normality of $G^{st}$ follows: let $g\in G$, then clearly $G/(G^{st})^{g}$ is also stable, whereby $G/(G^{st}\cap (G^{st})^{g})$ is stable.  The minimality of $G^{st}$ implies that 
$(G^{st})^{g} = G^{st}$. As $g\in G$ was arbitrary, $G^{st}$ is a normal subgroup of $G$, as required.

\end{proof}

\begin{Remark} (i) Theorem 1.1 implies the existence of $G^{00}$ (when $T$ has $NIP$). 
\newline
(ii)  The Lemmas above and Theorem 1.1  also hold when $G$ is type-definable.
\newline
(ii) We leave it to the reader to check that $G^{st}$ has no proper hyperdefinable stable quotient.
\newline
(iii) Also the proofs in this section adapt easily to showing that if $G$ is a stable hyperdefinable group in an arbitrary ambient theory, then any intersection of type-definable subgroups is a bounded subintersection. In other words there is no unbounded chain of type-definable subgroups of $G$.
\end{Remark}

%Question. In a stable theory, is any hyperdefinable group  *-definable.... Should be easy. By elimination of hyperimaginaries. 

\vspace{5mm}
\noindent

Finally we discuss some examples and ask some questions.  We return to the notation $G^{st,0}$ and $G^{st,00} (= G^{st})$ from the introduction.
As (in a $NIP$ theory) $G^{st}$ exists, so does $G^{st,0}$. An example where $G^{st,0}\neq G^{st,00}$ is any definably connected  definably compact group in an $o$-minimal expansion of $RCF$. In this case $G^{st,0} = G$ as we have $DCC$ on definable subgroups, and no quotient of $G$ by a proper definable subgroup can be stable. But $G^{st,00}$ is contained in $G^{00}$ which is a proper subgroup. We believe that in this $o$-minimal environment, and in fact in any distal theory,  $G^{00} = G^{st,00}$, but it requires checking. 
\begin{Problem} Give an example of a definable or even type-definable group $G$ in an $NIP$ theory such that $G = G^{00}$ but $G^{st,00}$ is properly contained in $G^{st,0}$.
\end{Problem}


\begin{thebibliography}{99}
\bibitem{Casanovas-book}  E. Casanovas, {\em Simple theories and hyperimaginaries}, Lecture Notes in Logic, Cambridge University Press -- Association of Symbolic Logic, 2011.
\bibitem{CH} Z. Chatzidakis and E. Hrushovski,  Model theory of difference fields, Transactions AMS, 351, number 8 (1999), p. 2997-3071.
\bibitem{HKP} B. Hart, B. Kim, A. Pillay, Coordinatisation and canonical bases in simple theories, JSL,  Volume 65, Issue 1 (2000), p. 293-309.
\bibitem{NIPI} E. Hrushovsk and A. Pillay, Groups, measures, and the $NIP$, Journal AMS, 21(2008), p. 563-596 .
\bibitem{Shelah} S. Shelah,  Minimal bounded index subgroup for dependent theories, Proceedings AMS,
136 (2008), 1087 -1091.



\end{thebibliography}
\end{document}